\documentclass[twoside,11pt]{amsart} 
\usepackage{epsfig}
\usepackage{latexsym} 
\usepackage{amsmath} 
\usepackage{amssymb}

\newtheorem{theorem}{Theorem}[section]
\newtheorem{lemma}[theorem]{Lemma}

\newdimen\epsfxsize

\newcommand {\gap}     {\makebox[0.075 in]{}}

\newcommand{\cons}[1] {\left<{#1}\right>}
\newcommand{\R}{\mathbb{R}}
\newcommand{\ER} {1 \times \R^d}
\newcommand {\set}[1]  {\left\{ {#1} \right\}}
\newcommand{\Wdg} {\omega}

\newcommand{\PWdg}{\tilde{\omega}}

\begin{document}

\title{Two observations on the perturbed wedge}
\author{Fred B. Holt}
\address{5520 - 31st Ave NE; Seattle, WA 98105; 
{\rm fbholt@earthlink.net}}

\date{DRAFT 23 Feb 2013 - posted ArXiv 23 May 2013}

\begin{abstract}
Francisco Santos has described a new construction, perturbing apart a non-simple face, to offer
a counterexample to the Hirsch Conjecture.  We offer two observations about this {\em perturbed wedge} construction, regarding its effect on edge-paths.  First, that an all-but-simple spindle of dimension $d$ and length $d+1$ is a counterexample to the nonrevisiting conjecture.  Second, that there are conditions under which the perturbed wedge construction does not increase the diameter.

\vskip .0325in

NOTE:  These are simply working notes, offering two observations on the construction identified
by Santos.
\end{abstract}


\maketitle


\section{Introduction}
Coincidence in high dimensions is a delicate issue.  Santos has brought forward the
perturbed wedge construction \cite{Paco}, to produce a counterexample to
 the Hirsch conjecture.  We start in dimension $5$ with a non-simple 
 counterexample to the nonrevisiting conjecture.  If this counterexample were simple, 
 then repeated wedging would produce a corresponding counterexample to the 
 Hirsch conjecture.  Since our counterexample to the nonrevisiting conjecture is
 not simple, we need an alternate method to produce the corresponding counterexample to
 the Hirsch conjecture, and the perturbed wedge provides this method.

A $d$-dimensional spindle $(P,x,y)$ is a polytope with two distinguished 
vertices $x$ and $y$ such that every facet of $P$ is incident to either $x$ or $y$.
The {\em length} of the spindle $(P,x,y)$ is the distance $\delta_P(x,y)$.
The spindle $(P,x,y)$ is {\em all-but-simple} if every vertex of $P$ other than $x$ and $y$ is 
a simple vertex.

Our first observation is that a $d$-dimensional all-but-simple spindle $(P,x,y)$ 
of length $d+1$ is a counterexample to the nonrevisiting conjecture.

Let $P$ be a $d$-dimensional polytope with $n$ facets.  Let $y$ be a nonsimple vertex of $P$ incident to a facet $G$, and let $F$ be a facet not incident to $y$.  The {\em perturbed wedge}
$\tilde{\omega}_{F,G}P$ of $P$ is
a wedge of $P$ with foot $F$, which is a $(d+1)$-dimensional polytope with $n+1$ facets,
followed by a perturbation of the image of the facet $G$ in the coordinate for the new 
dimension.

Our second observation is that for a spindle $(P,x,y)$, with a facet $F$ incident to $x$ and a facet
$G$ incident to $y$, if there is a nonsimple edge in $G$ from $y$ to a vertex in $F$, then a short path from $x$ to $y$ along this edge is not increased in length in $\tilde{\omega}_{F,G} P$.

\section{A counterexample to the nonrevisiting conjecture}

\begin{lemma} Let $(P,x,y)$ be a $d$-dimensional all-but-simple spindle of length $d+1$.
Then every path from $x$ to $y$ revisits at least one facet.
\end{lemma}

\begin{proof}
Let $X$ be the $n_1$ facets incident to $x$, and let $Y$ be the $n_2$ facets incident to $y$.
Let $\rho = [x,u_1,\ldots,u_{k-1},y]$ be a path from $x$ to $y$ of length $k > d$, with all of the
$u_i$ being simple vertices.

Then each $u_i$ is incident to $d$ facets.  $u_1$ is incident to $d-1$ facets in $X$ and one
facet in $Y$, and $u_k$ is incident to one facet in $X$ and $d-1$ facets in $Y$.  The incidence
table for $\rho$ looks like the following:
$$
\begin{array}{llcccccc|cccccc}
 & \multicolumn{6}{c}{X} & \multicolumn{6}{c}{Y}  \\ \hline
x: & n_1 \; {\rm vertices} & 1 & 1 & \cdots & 1 & \cdots  & 1 & & & & & &  \\
u_1: & d \; {\rm vertices}  &  &  &  & 1 & \cdots & 1 & 1 & & & & &  \\
  & &  & &  &  & \cdots &  & & \cdots & & & &  \\
u_{k-1}: &  d  \; {\rm vertices}  & &  &  &  &  & 1 & 1 & \cdots & 1 &  &  &   \\
y: &  n_2 \; {\rm vertices}  & &  &  &  &  &  & 1 & \cdots & 1 & \cdots & 1 & 1  \\ \hline
\end{array}
$$

Consider the facet-departures and facet-arrivals from $u_1$ to $u_k$ (the simple part of the path).  
If any of the arrivals were back
to a facet in $X$, this would be a revisit since these facets were all incident 
to the starting vertex $x$.  So the arrivals must all be in $Y$.  

Similarly, all of the departures must be from $X$.  Any departure from a facet in $Y$ would
create a revisit since all the facets in $Y$ are incident to the final vertex $y$.

There are too many arrivals and departures to prevent a revisit.
The vertex $u_1$ is incident to only  $d-1$ facets in $X$, and since each departure leaves
a facet of $X$, $u_j$ is incident to $d-j$ facets in $X$.  So $u_d$ has completely departed
from $X$.  Since $k > d$, $u_d$ occurs among the vertices $u_1, \ldots, u_{k-1}$, but since
$u_d$ is incident only to facets in $Y$, $u_d$ must already be the vertex $y$, and $\rho$ would 
have length at most $d$.
\end{proof}

Santos \cite{Paco} has produced all-but-simple spindles in dimension $5$ of length $6$. 
Currently the smallest example has 25 facets.
So in dimension $5$ with $25$ facets, we have a counterexample to the nonrevisiting conjecture, with length well below the Hirsch bound. Since the spindle is not simple, at $x$ and $y$, 
the usual way of generating the corresponding counterexample \cite{KW, Hnr} to the Hirsch conjecture (through repeated wedging) does not directly apply, and we need an alternate construction.  The perturbed wedge accomplishes this.

\section{The perturbed wedge}
The perturbed wedge is constructed in two steps, first as a wedge over a facet, followed
by a perturbation of a facet incident to a nonsimple vertex of the wedge.

Let $P$ be a $d$-dimensional polytope with $n$ facets and $m$ vertices, and let $F=F(u)$ be 
a facet incident to $x$ in $P$.
The wedge $W = \omega_F(P)$ is a $(d+1)$-dimensional polytope with $n+1$ facets
and $2m - f_0(F)$ vertices.
The wedge $\Wdg_F P$ over $F$ in $P$, corresponds to the two-point
suspension over $u$ in $P^*$.

\begin{figure}[tb] 
\centering
\includegraphics[width=4.5in]{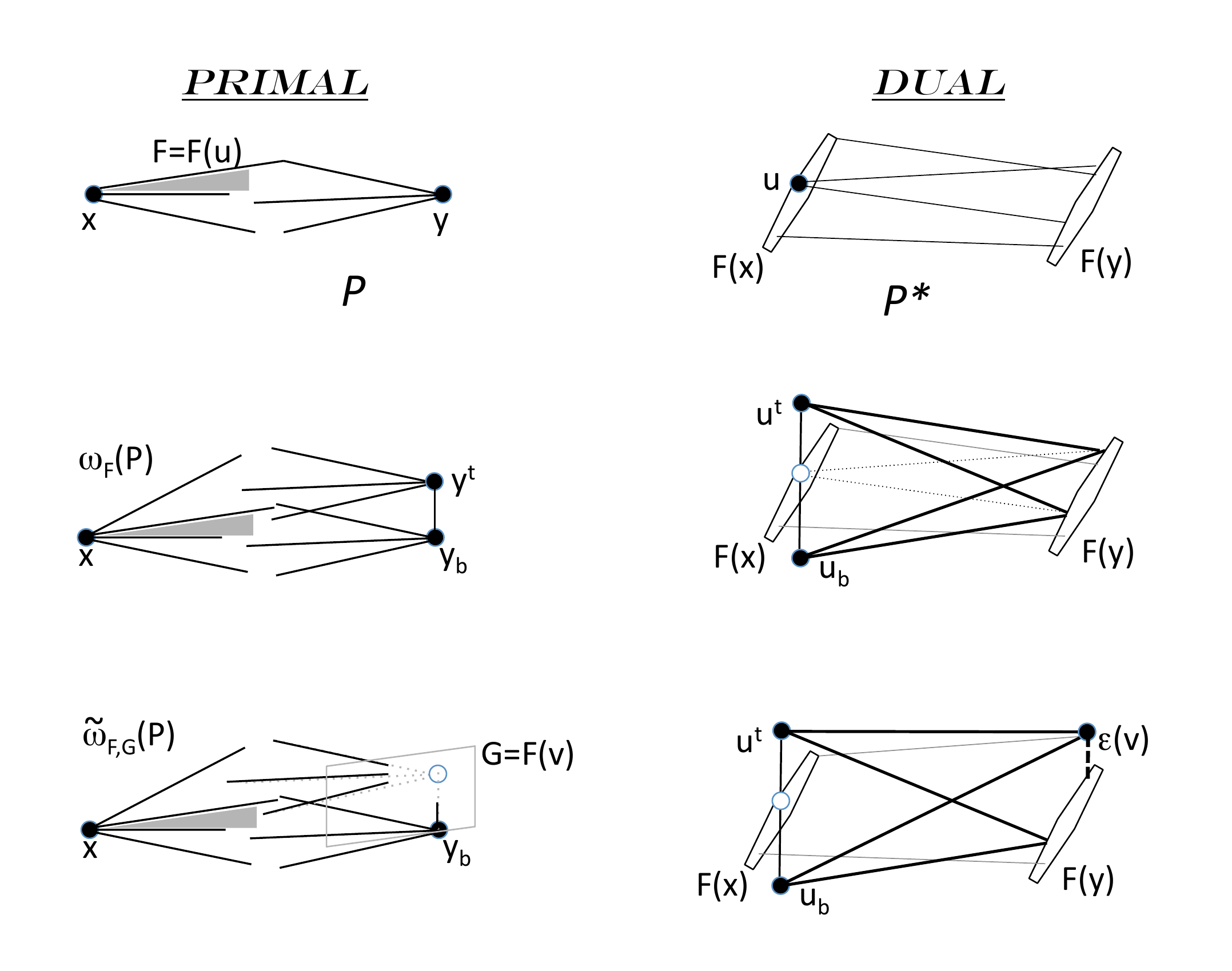}
\caption{\label{SantosWedgeFig} This figure illustrates the Santos perturbed wedge 
construction and its dual construction.  In the primal setting, we first perform a wedge
of $P$ over a facet $F$ which is incident to $x$, followed by a vertical perturbation of
a facet $G$ incident to $y$.}
\end{figure}

We now perturb a facet $G=F(v)$ in $W$ incident to the edge $[y^t, y_b]$.  This is already
interesting; we don't encounter non-simple edges until dimension $4$.
The perturbation of $G$ is accomplished by introducing a small vertical displacement
to $v$; that is, instead of the last coordinate of the outward-pointing normal vector to $G$ 
being $0$, we perturb this coordinate to $\epsilon > 0$.

\subsection{Embedded construction}
As a canonical embedding for polytopes, we consider the vertices of $P$
to be embedded in $\ER$, with $0$ in the interior of $P$.  For the facets
of $P$ we take their outward-pointing normals.  Since $0$ is interior to $P$,
we can assume that the first coordinate of each normal is $-1$.  $P$ is
given by the embedding:
$$ H^T_{n \times (d+1)} V_{(d+1) \times m}  \le \cons{0},$$
with
$$ \left[ -1 \gap h_i^T \right] \cdot \left[ \begin{array}{c} 1 \\ x_j \end{array} \right] = 0$$
iff vertex $j$ is incident to facet (hyperplane) $i$.
The facet-vertex incidence matrix for $P$ is given by the $\{0,1\}$-matrix
$$M_{n \times m}(P) = {\rm IsZero} \left( H^T V \right).$$

As a $d$-dimensional polytope, each vertex of $P$ is incident to at least
$d$ facets, and each $k$-face of $P$ is incident to at least $k+1$ vertices.
A simple vertex is incident to exactly $d$ facets.

For a $k$-face of $P$, each incident facet contributes either to the affine space
supporting this face or to its boundary \cite{Hblend}.
The space supporting this $k$-face is the intersection of at least $d-k$ facets of $P$,
and we say that the space is {\it simple} iff this space is given by the
coincident intersection of exactly $d-k$ facets of $P$. 
For $k > 0$, the boundary of the $k$-face is created by the various intersections of at least $k+1$
other facets with the supporting space of the face.
A face is simple iff its space is simple and all of its boundary elements are simple.
A face can be nonsimple in a variety of ways or in multiple ways, through the nonsimplicity
of its space or of its various boundary elements.

Let $F$ be represented by the first outward-pointing normal in $H^T$, and $G$ by the
last outward-pointing normal.  Then $H^T(\Wdg_F P)$ is given canonically by
$$ H^T(\Wdg_F P) = \left[  \begin{array}{ccc}
-1 & h_1^T & 1 \\
-1 & h_1^T & -1 \\
-1 & h_2^T & 0 \\
\vdots & \vdots & \vdots \\
-1 & h_n^T & 0 \end{array} \right].$$
The facet $F$ is replaced by two facets, the top and the base of the wedge.
The top has final coordinate $1$, and the base $-1$.
Every other facet is replaced by a single {\it vertical} facet, meaning that the last
coordinate (the new coordinate) is $0$.

The vertices of $\Wdg_F P$ are given as follows.  Rearrange the columns of $V$
so that the vertices incident to $F$ are given in the first block $V_F$ and
the rest of the vertices occur in a second block $V_{-}$ 
(denoted this way since $[-1 \; h_1^T]\cdot V_{-} < \cons{0}$).
\begin{eqnarray*}
V(P)  & =  & \left[ \begin{array}{cc} V_F & V_{-} \end{array} \right]. \\
  {\rm and} & &  \\
 V(\Wdg_F P) & = & \left[ \begin{array}{ccc}
 V_F & V_{-} & V_{-} \\
 \cons{0} & -[-1 \; h_1^T]\cdot V_{-} & [-1 \; h_1^T]\cdot V_{-}
 \end{array} \right].
 \end{eqnarray*}
 In $V(\Wdg_F P)$, the vertices are now embedded in $1 \times \R^{d+1}$.
 The last coordinate for vertices in the foot $F$ is $0$, and for vertices not in the
 foot, there are two images, one in the top and one in the base.

Let the facet $G$ have outward-pointing normal $[-1 \; h_n^T \; 0]$.  We perturb
the last coordinate to $\epsilon > 0$ to complete the construction of the perturbed
wedge.
$$ H^T(\PWdg_{F,G} P) = \left[  \begin{array}{ccc}
-1 & h_1^T & 1 \\
-1 & h_1^T & -1 \\
-1 & h_2^T & 0 \\
\vdots & \vdots & \vdots \\
-1 & h_{n-1}^T & 0 \\
-1 & h_n^T & \epsilon \end{array} \right].$$
Denote the outward-pointing normal for $\tilde{G}$ by 
$h_{\tilde{G}}^T = [-1 \; h_n^T \; \epsilon].$

To understand the effect of perturbing the facet $G$, we consider both $G$ and its
perturbed image $\tilde{G}$.  While we were able to write down the vertices of the
wedge $\Wdg_F P$ explicitly, the effect of the perturbation is more complicated.
The vertices incident to the facet $G$ are of three types:
\begin{itemize}
\item[Foot:] $v \in G \cap T \cap B$ (or $v \in G \cap F$).  For these vertices the last coordinate
is $0$, so $h_{\tilde{G}}^T v = 0$, and these vertices remain after the perturbation.

\vskip .0625in

\item[Top:] $v \in G \cap T  \backslash B$.  For these vertices, the last coordinate is positive,
so $h_{\tilde{G}}^T v > 0$.  If $v$ consists combinatorially of a single edge terminated by
$G$ -- the case when $v$ is a simple vertex but also when $v$ consists of a single 
nonsimple edge terminated by $G$ -- then the vertex $v$ is perturbed back along this
edge.  If $v$ consists combinatorially of more than one edge being terminated by $G$,
then $v$ is truncated away by the perturbation, and $\tilde{G}$ introduces vertices along
all of the edges incident to $v$ but not lying in $G$.

\vskip .0625in

\item[Base:] $v \in G \cap B  \backslash T$.  For these vertices, the last coordinate is negative,
so $h_{\tilde{G}}^T v < 0$.  If $v$ consists combinatorially of a single edge terminated by
$G$ -- the case when $v$ is a simple vertex but also when $v$ consists of a single 
nonsimple edge terminated by $G$ -- then the vertex $v$ is perturbed out along this
edge.  If $v$ consists combinatorially of more than one edge being terminated by $G$,
then $v$ remains as a vertex of $\PWdg_{F,G}P$,
and the perturbation reveals new edges emanating from $v$ and terminating in 
$\tilde{G}$ in new vertices.
\end{itemize}

We now consider the effect of the perturbed wedge on vertex $y$ and its natural images.

Under the wedge, $y$ has two natural images $y^t$ and $y_b$, in the top and base respectively.
Since $y$ is a nonsimple vertex, $y^t$ and $y_b$ are nonsimple vertices, and the edge
$[y_b,y^t]$ between them is a nonsimple edge.  The facet $G$ is one of the facets supporting 
the space of this edge.

When $G$ is perturbed to $\tilde{G}$, $y_b$ is preserved as a vertex, $y^t$ is truncated away,
and $\tilde{G}$ terminates the vertical edge at a new vertex $y_0$ whose last coordinate is $0$.
$\tilde{G}$ introduces new vertices along the edges of $\Wdg_F P$ incident to $y^t$ but not
lying in $G$.  $\tilde{G}$ also introduces new vertices along the new edges emanating from $y_b$
as revealed by $\tilde{G}$.  That is, the collection of facets $Y$ and the facet $B$ intersect in
edges that had lain beyond the facet $G$.  $\tilde{G}$ now introduces these edges as part of the
boundary of $\PWdg_{F,G}P$ and terminates them in new vertices.

In considering the implications of the perturbed wedge construction on the Hirsch conjecture,
we are interested in short paths from $x$ to $y$ in $P$ and their tight natural images
from $x$ to $y_0$ in $\PWdg_{F,G}P$.

{\em Claim:} The perturbed wedge does not introduce revisits on tight natural images of short paths.  
The detailed study of this claim is beyond the scope of this note.

For the construction of the counterexample to the Hirsch conjecture, the revisits already exist 
in the initial $5$-dimensional spindle.  We see below that as a general construction, the perturbed wedge does not always increase the length of the input polytope by $1$.  
However, the repeated application of the perturbed wedge to an all-but-simple spindle avoids the conditions of the following lemma.

\begin{figure}[tb] 
\centering
\includegraphics[width=4.5in]{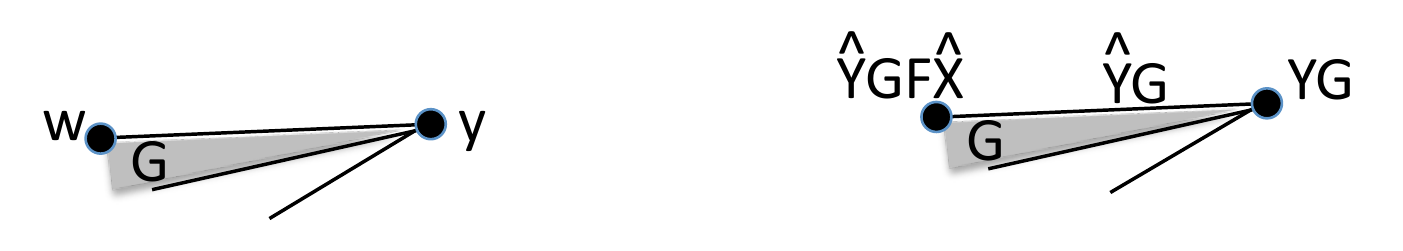}
\caption{\label{FacetFig} Consider the facet incidences in the circumstance that
$y$ has a neighbor $w$ along an edge of $G$ that terminates in the facet $F$, which
will be the foot of the wedge.  Denote the facet incidences at the nonsimple vertex $y$ 
as $YG$, in which $Y$ is a set of at least $d$ facets.  The edge from $y$ to $w$ is the
coincidence of $\hat{Y}G$ in which at least one facet of $Y$ is omitted from $\hat{Y}$.
The vertex $w$ is coincident with the facets $\hat{Y}G$ and $F$ and perhaps additional
facets $\hat{X}$.  $\hat{X}$ may be empty.}
\end{figure}

\begin{lemma}  Let $y$ be a nonsimple vertex of a $d$-dimensional polytope $P$.  Let $F$
be a facet of $P$ not incident to $y$, and let $G$ be a facet incident to $y$.
 If there is a nonsimple edge in $G$ from $y$ to a vertex $w$ in $F \cap G$, then this edge 
 remains after the perturbation.
\end{lemma}

\begin{figure}[tb] 
\centering
\includegraphics[width=4.5in]{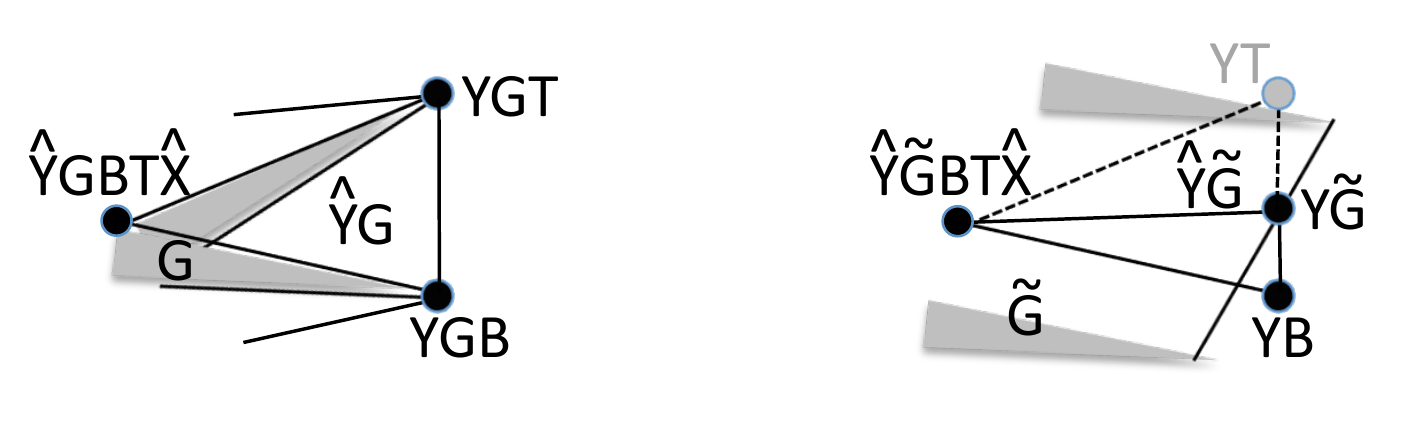}
\caption{\label{NbrVtxFig} Consider the action of the perturbed wedge on the
facet-coincidences at $y$ and $w$.  Under the wedge, $YG$ becomes an edge and
$\hat{Y}G$ becomes a $2$-face.  The vertex $y$ has two natural images, $y^t$ which is 
incident to the facets $YGT$, and $y_b$ which is incident to $YGB$.  The facet $F$ is 
replaced by two facets, the top $T$ and the base $B$.
Now we perturb the facet $G$, introducing a small positive value in the last coordinate of its outward
normal.
Since $y$ was nonsimple, the vertex $YB$ remains, but the vertex $YT$ is truncated away.
Instead, $\tilde{G}$ now intersects the vertical edge $Y$ in the plane 
$\set{1}\times \R^d \times \set{0}$.  The $2$-face $\hat{Y}G$  was nonsimple, and so the
$2$-face remains with its space supported by $\hat{Y}$, and $\tilde{G}$ intersects
it in an edge $[w,y_0]$.
}
\end{figure}

\begin{proof}
As a nonsimple edge, the 1-dimensional space of $[w,y]$ is defined by 
the coincidence of $G$ and at least $d-1$ other facets $\hat{Y}$.  
See Figure~\ref{FacetFig}. 
These facets $\hat{Y}G$
are incident to both $w$ and $y$.  The boundary of the edge at $w$ is established by
$F$ and possibly more facets $\hat{X}$, none of which can be incident to $y$.  The boundary
of the edge at $y$ is established by the facets $Y\setminus \hat{Y}$.

Under the wedge $\Wdg_F P$, the image of the edge is a nonsimple $2$-dimensional face,
the triangle with vertices $w$, $y^t$, and $y_b$.

Now, under the perturbation, the space of the triangular face is still defined by $\hat{Y}$.
$\tilde{G}$ intersects this face in the plane with last coordinate $0$, creating an edge from
$w$ to $y_0$.
\end{proof}


In the dual setting, the wedge over the facet $F(u)$ corresponds to a
two-point suspension $S_u(P^*)$ over the vertex $u$.  See Figure~\ref{SantosWedgeFig}.
The perturbation of the facet $G=F(v)$ corresponds to a vertical perturbation of the vertex $v$.
The observation in the previous lemma is that if the facet $Y$ is adjacent to a facet $W$, such
that $u,v \in W$ and that the ridge $Y \cap W$ is not simplicial, then after the two-point suspension
over $u$ and the perturbation of $v$, the new $Y$ is still adjacent to $W$ across the ridge, with
$v$ removed from the ridge.


Since the spindles used to seed the construction of the counterexample to the Hirsch conjecture
are all-but-simple, the nonsimple edge does not occur on short paths between $x$ and $y$.
Although the natural images of $x$ and $y$ under this construction are not simple vertices, 
and although they are connected by nonsimple edges, the edges from $y_0$ run to natural
images of $y$ on the bases of various wedges, which should not be reused as the foot
of the wedge.  Thus the condition on $G \cap F$ does not occur for short paths.

\section{Summary}
Santos' construction of the first known counterexample to the Hirsch conjecture, for bounded
polytopes, follows the strategy of first finding a counterexample to the nonrevisiting conjecture.
Santos constructs a $5$-dimensional all-but-simple spindle $(P,x,y)$ of length $6$, which is
a counterexample to the nonrevisiting conjecture.

For simple polytopes, if we had a counterexample to the nonrevisiting conjecture, we would
produce the corresponding counterexample to the Hirsch conjecture through repeated wedging,
over all the facets not incident to $x$ or $y$.  However, Santos $5$-dimensional spindle is not simple.  Every facet is incident to either $x$ or $y$,
so we need an alternate method to produce the corresponding counterexample to the Hirsch
conjecture.  The perturbed wedge accomplishes this.


\bibliographystyle{alpha}

\bibliography{../Biblio/poly}

\begin{thebibliography}{KW67}

\bibitem[Hol03]{Hnr}
F.B. Holt.
\newblock Maximal nonrevisiting paths in simple polytopes.
\newblock {\em Discrete Mathematics}, (1-3):105--128, 2003.

\bibitem[Hol04]{Hblend}
F.B. Holt.
\newblock Blending simple polytopes at faces.
\newblock {\em Discrete Mathematics}, (1-3):141--150, 2004.

\bibitem[KW67]{KW}
V.~Klee and D.W. Walkup.
\newblock The $d$-step conjecture for polyhedra of dimension $d<6$.
\newblock {\em Acta Mathematica}, 117:53--78, 1967.

\bibitem[San10]{Paco}
F.~Santos.
\newblock A counterexample to the {H}irsch {C}onjecture.
\newblock {\em arXiv}, 1006.2814v1:1--27, 2010.

\end{thebibliography}

\end{document}